\pdfoutput=1
\documentclass{amsart}
\usepackage{booktabs}
\usepackage{amssymb,}
\usepackage[applemac]{inputenc}
\usepackage[english]{babel}
\makeindex

\usepackage{hyperref}
\hypersetup{
colorlinks=true,       
    linkcolor=blue,          
    citecolor=blue,        
    filecolor=blue,      
    urlcolor=blue           
}

\usepackage[all]{xy}
\usepackage{tikz}
\usepackage{todonotes}
\usepackage[alphabetic, backrefs, lite]{amsrefs}
\usepackage{enumerate}
\usepackage{verbatim}
\usepackage{multirow}
\usepackage{pgf, tikz}
\usepackage{float}
\usepackage{rotating}
\usetikzlibrary{arrows}

\CompileMatrices

\newtheorem{theorem}{Theorem}[section]
\newtheorem{lemma}[theorem]{Lemma}

\newtheorem{proposition}[theorem]{Proposition}
\newtheorem{corollary}[theorem]{Corollary}
\theoremstyle{definition}

\newtheorem{example}[theorem]{Example}

\newtheorem{remark}[theorem]{Remark}

\def\C{{\mathbb C}}

\def\Z{{\mathbb Z}}
\def\R{{\mathbb R}}

\def\P{{\mathbb P}}

\def\Pic{\operatorname{Pic}}

\def\rk{\mbox{rank}}

\def\tr{\mbox{tr}}
\def\Fix{\mbox{Fix}}

\begin{document}

\title{Order $9$ automorphisms of K3 surfaces}

\author{Michela Artebani}
\address{
Departamento de Matem\'atica, \newline
Universidad de Concepci\'on, \newline
Casilla 160-C,
Concepci\'on, Chile}
\email{martebani@udec.cl}

\author{Paola Comparin}
\address{
Departamento de Matem\'atica y Estad\'istica, \newline
Universidad de la Frontera, \newline
Temuco, Chile}
\email{paola.comparin@ufrontera.cl}

\author{Mar\'ia Elisa Vald\'es}
\address{
Departamento de Matem\'atica, \newline
Universidad de Concepci\'on, \newline
Casilla 160-C,
Concepci\'on, Chile}
\email{mariaevaldes@udec.cl}

\thanks{The first and last author have been partially 
supported by Proyecto FONDECYT Regular 
N. 1160897 and Proyecto Anillo CONICYT PIA  ACT 1415. 
The second author has been partially supported by  
Proyecto FONDECYT Postdoctorado N. 3150015 and
Proyecto Anillo ACT 1415 PIA CONICYT}

\begin{abstract}
In this paper we provide a complete classification of 
non-symplectic automorphisms of order nine 
on complex K3 surfaces.
\end{abstract}
\maketitle

\section*{Introduction}
An automorphism of order $n$ of a complex projective K3 surface is called {\em non-symplectic} 
if its action on the vector space of holomorphic $2$-forms is non-trivial,
 {\em purely non-symplectic} if such action has order $n$. 
By \cite[Theorem 0.1]{Nikulin} the rank of the transcendental lattice of a K3 surface carrying a  purely non-symplectic 
automorphism of order $n$ is divisible by the Euler's function of $n$.
This implies that $\varphi(n)\leq 21$, and all positive integers $n\not=60$ with such property 
occur as orders of purely non-symplectic automorphisms by \cite[Main Theorem 3]{MO}.   
A  classification of purely non-symplectic automorphisms 
is known for all prime orders  \cite{Nikulin, OZ1, OZ2, V, OZ3, Ko, ArtebaniSarti, ArtebaniSartiTaki},
when $\varphi(n)=20$ \cite{MO}, when the automorphism acts trivially on the N\'eron-Severi lattice 
and $\varphi(n)$ equals the rank of the transcendental lattice \cite{V, Ko, OZ3, Schuett}, 
for orders $6,16$ \cite{Dillies, TabbaDimaSarti} 
and $4,8$ \cite{ArtebaniSarti4, TabbaaSarti} (the latter contain partial classifications).

In the present paper we classify non-symplectic automorphisms of order $9$, 
completing previous work by Taki \cite{Taki}, who studied the case when 
the automorphism acts trivially on the N\'eron-Severi lattice. 
The main result is the following.

\begin{theorem}\label{main}
Let $X$ be a complex K3 surface, $\sigma$ be a non-symplectic automorphism of $X$ of order nine and $\tau=\sigma^3$.
Then 
\begin{enumerate}
\item $\sigma$ is purely non-symplectic, i.e. $\tau$ is non-symplectic;
\item 
the topological structure of  $\Fix(\tau)$, $\Fix(\sigma)$, the ranks of the eigenspaces of $\sigma^*$ and $\tau^*$ in $H^2(X,\C)$ 
and the invariant lattice of $\tau$ are described in Table \ref{tabla1} (see Section \ref{section1} for the notation). 
\end{enumerate}
All configurations described in Table \ref{tabla1} exist.
\end{theorem}

The irreducible components of the moduli space of K3 surfaces 
carrying a purely non-symplectic automorphism 
of order $n>2$ are known to be complex ball quotients \cite{DK}.
In some cases these spaces have interesting interpretations 
as moduli spaces of other geometric objects \cite[\S 12]{DK}.
When $n=9$ we show that the moduli space has three $2$-dimensional components
and each of them is birational to a moduli space of curves with cyclic automorphisms 
of order three.

\begin{table}[h]
\begin{center}
\begin{tabular}{cccc||ccc}
\multicolumn{4}{c}{}\\
&&$\tau$&&&$\sigma$&\\
\hline\hline
Case  & $(n,k,g)$ & $m$ & $L^\tau$ & Case &$(n_\sigma,k_\sigma,g_\sigma)$&$(r,l)$\\
\hline\hline
\multirow{2}{*}{A}&\multirow{2}{*}{(1,1,3)}&\multirow{2}{*}{$9$} & \multirow{2}{*}{$U(3)\oplus A_2$}
     &A1&$(6,0,-)$&(4,0)\\
&&&&A2&$(3,0,-)$&(2,1)\\
\hline\hline

\multirow{1}{*}{B}&\multirow{1}{*}{(1,2,4)}&$9$ & \multirow{1}{*}{$U\oplus A_2$}
&B&$(6,0,-)$& (4,0)\\
  
\hline\hline

\multirow{1}{*}{C}&\multirow{1}{*}{(4,1,0)}&$6$ & \multirow{1}{*}{$U(3)\oplus A_2^4$}
&C&$(3,0,-)$&(4,3)\\

\hline\hline

\multirow{4}{*}{D}&\multirow{4}{*}{(4,2,1)}&\multirow{4}{*}{$6$} & \multirow{4}{*}{$U\oplus A_2^4$}
     &D1&$(7,1,0)$& (8,1)\\
&&&&D2&$(3,1,1)$&(4,3)\\
&&&&D3&$(6,0,-)$&(6,2)\\
&&&&D4&$(3,0,-)$&(4,3)\\
\hline\hline

\multirow{1}{*}{E}&\multirow{1}{*}{(4,3,2)}&\multirow{1}{*}{$6$} & \multirow{1}{*}{$U\oplus E_6\oplus A_2$}
    &E&$(10,1,0)$& (10,0)\\
 
\hline\hline

\multirow{1}{*}{F}&\multirow{1}{*}{(4,4,3)}&\multirow{1}{*}{$6$} & \multirow{1}{*}{$U\oplus E_8$}
    &F&$(10,1,0)$& (10,0)\\
\hline\hline

\multirow{2}{*}{G}&\multirow{2}{*}{(7,4,0)}&\multirow{2}{*}{$3$}& \multirow{2}{*}{$U\oplus E_6^2\oplus A_2$}
&G1&$(10,1,0)$&(12,2)\\  
 &  &&&G2&$(3,0,-)$&(6,5) \\

\hline\hline

\multirow{1}{*}{H}&\multirow{1}{*}{(7,5,1)}&\multirow{1}{*}{$3$} & \multirow{1}{*}{$U\oplus E_6\oplus E_8$}
    &H&$(14,2,0)$& (16,0)\\
\ 

\end{tabular}
{\small{\caption{Classification of automorphisms of order $9$ of K3 surfaces}\label{tabla1} }}
\end{center}
\end{table}

\section{Preliminaries} \label{section1}

\subsection{Non symplectic automorphisms}
We start recalling some known results about non-symplectic automorphisms 
on complex $K3$ surfaces. We will denote by $\zeta_n$ 
 a primitive $n$-th root of unity.
Given an automorphism $f$ of $X$, its fixed locus is
\[
\Fix(f)=\{x\in X: f(x)=x\}. 
\]
The following result holds for any automorphism, symplectic or not (see \cite[\S 5]{Nikulin}).
\begin{lemma}\label{diag}
Let $X$ be a $K3$ surface, $f$ be an automorphism of order $n$ on $X$ 
such that $f^*(\omega_X)=\zeta_n^k\omega_X$, where $k\in \{1,\dots,n\}$,
and $p\in X$ be a fixed point for $f$. 
Then $f$  can be locally linearized and diagonalized in a neighborhood of $p$ 
so that is given by a matrix of the form:
\[
\begin{array}{ccc}
A_{i,j}=\begin{pmatrix}
\zeta^{i}&0\\
0&\zeta^j
\end{pmatrix}, &\mbox{with } i+j\equiv k\ ({\rm mod}\, n),\ 1\leq i,j\leq n.
\end{array}
\]
\end{lemma}

It follows from this result that the fixed locus of $f$
is the disjoint union of smooth curves and isolated points.
Moreover, by Hodge index theorem, $\Fix(f)$ contains 
at most one curve of positive genus and at most two disjoint 
curves of genus $1$.

We also recall the following useful result, which tells how a 
non-symplectic automorphism acts on a tree of rational curves 
(see for example \cite[Lemma 4]{ArtebaniSarti4}). 

\begin{lemma}\label{tree}
Let $T =\sum_iR_i$ be a tree of smooth rational curves on a $K3$ surface 
$X$ such that each $R_i$ is invariant under the action of a purely non-symplectic automorphism 
$f$ of order $n$. Then, the points of intersection of the rational curves 
$R_i$ are fixed by $f$ and the action at one fixed point determines the action on the whole tree.
If  $f^*(\omega_X)=\zeta_n\omega_X$, then the sequence of fixed points is either 
\[\ldots\  A_{1,0},\ A_{1,0},\ A_{2,n-1},\ldots A_{\frac{n}2,\frac{n}2+1},\ A_{\frac{n}2,\frac{n}2+1},\ \ldots\quad \text{if n is even, or}
\]
\[\ldots\ A_{1,0},\ A_{1,0},\ A_{2,n-1},\ldots A_{\frac{n+1}2,\frac{n+1}2},\ A_{n-1,n},\ \ldots\quad \text{if n is odd}.
\]
\end{lemma}

An automorphism $f$ induces an isometry $f^*$ on $H^2(X,\Z)$.
In what follows we will denote by $L^{f}$ the sublattice of $H^2(X,\Z)$ invariant for $f^*$.
Observe that, when $f$ is non-symplectic, $L^{f}$ is contained in $\Pic(X)$.

\subsection{Order nine}
Let $\sigma$ be a non-symplectic automorphism of order nine of a $K3$ surface $X$ and 
let $\tau=\sigma^3$. We will start proving that $\tau$ is non-symplectic.
In this section we will denote $\zeta=\zeta_9$.

\begin{lemma}\label{tauns}
Let $X$ be a $K3$ surface. If $\sigma$ is an order 9 non-symplectic automorphism of $X$, 
then $\tau=\sigma^3$ is non-symplectic.  
\end{lemma}

\begin{proof}  Assume that $\tau$ is symplectic i.e.  $\sigma^*(\omega_X)=\zeta^3 \omega_X$.
Since $\tau$ has only isolated fixed points  \cite[\S 5]{Nikulin},
the same is true for $\sigma$.
Let $p$ be a fixed point of $\sigma$. 
By Lemma \ref{diag} $\sigma$ can be locally linearized and diagonalized  to  be of the form
\[
A_{i,j}=\begin{pmatrix}\zeta^{i}& 0\\0& \zeta^{j}\end{pmatrix},\qquad \mbox{with }  i+j\equiv 3\ ({\rm mod}\, 9),
\]
thus the possible types are $A_{1,2}, A_{4,8}$ and $A_{5,7}$.
Let $a_{1,2}, a_{4,8}, a_{5,7}$ be the number of  points of types $A_{1,2},A_{4,8}, A_{5,7}$ respectively. 
By the holomorphic Lefschetz's formula \cite[Theorem 4.6]{AS} we have 
\[
1+\zeta^{-3}= \frac{a_{1,2}}{(1-\zeta)(1-\zeta^2)}+\frac{a_{4,8}}{(1-\zeta^4)(1-\zeta^8)}+\frac{a_{5,7}}{(1-\zeta^5)(1-\zeta^7)},
\]
which  is inconsistent. Therefore, $\tau$ is non-symplectic. 
\end{proof}

In \cite[Theorem 2.2]{ArtebaniSarti} the authors classified 
order $3$ non-symplectic automorphisms $\tau$ of $K3$ surfaces relating
the structure of their fixed locus with their action in cohomology.
In particular they proved that 
\[
\Fix(\tau)= \{p_1,\dots, p_{n}\} \sqcup C_{g} \sqcup E_1\sqcup\dots \sqcup E_{k-1},
\]
 where the $E_i$'s are smooth curves of genus $0$, $C_g$ is a smooth curve of genus $g$ and
 moreover $m+n=10$, where $2m=22-\rk(L^{\tau})$. 
 Since   the eigenvalues of $\sigma^*$  in  $(L^{\tau})^{\perp}\otimes_{\Z} \C\subseteq H^2(X,\C)$  
are the primitive $9$-th roots of unity,  then  $2m$ is divisible by $\varphi(9)=6$.
This implies that $m\in \{3,6,9\}$ and by \cite[Table 1]{ArtebaniSarti} the fixed locus of
$\tau$ is described in the following table.

 \begin{table}[H]
\begin{tabular}{c|ccc|c|}
Case &$(n,k,g)$&$m$\\
\hline\hline
A & (1,1,3) &9\\
\hline
B& (1,2,4) & 9\\
\hline
C& (4,1,0)& 6\\
\hline
D& (4,2,1) & 6\\
\hline
E& (4,3,2) &6\\
\hline
F& (4,4,3) &6\\
\hline
G& (7,4,0) &3\\
\hline
H& (7,5,1) &3\\
\end{tabular}
\vspace{0.2cm}
\caption{Fixed locus of $\sigma^3$}
\label{tau}
\end{table}

We will denote by $n_{\sigma}$ 
the number of isolated points in $\Fix(\sigma)$, by $g_{\sigma}$ the maximal genus 
of a curve in it (if any) and by $k_\sigma$ the number of curves in it.
By Lemma \ref{diag}  the fixed points of $\sigma$ are of one of the following  types: 
\[
A_{1,0},\ A_{2,8},\ A_{3,7},\ A_{4,6},\ A_{5,5}.
\]
The points of type $A_{1,0}$ are those which are contained in a fixed curve by $\sigma$,
those of type $A_{3,7}$ and $A_{4,6}$ are contained in fixed curve by $\tau$ 
and the points of the remaining types are isolated for $\sigma$ and $\tau$.
We will denote by $a_{i,j}$ the number of fixed points of type $A_{i,j}$.

\begin{lemma}\label{lemaA}
Let $
\alpha=\displaystyle\sum_{C\subset {\tiny{\Fix(\sigma)}}}(1-g(C)).
$
The following relations hold:
\begin{equation}\label{star}
\begin{cases}
a_{2,8}+a_{5,5}  = 3\alpha+1\\
a_{3,7} = 2\alpha+1\\
a_{4,6}+3a_{2,8}= 8\alpha+4.
\end{cases}
\tag{$*$}\end{equation}
In particular $\alpha\geq 0$.
Moreover $n_{\sigma}+2\alpha=2+r-l$, where $r=\rk(L^{\sigma})$ and
$
l=\dim\{v\in H^{2}(X,\C):\sigma^*(v)=\zeta^3 v\}.
$
\end{lemma}

\begin{proof}
By the holomorphic Lefschetz's formula \cite[Theorem 4.6]{AS} 
we have that:
\[
1+\overline{\zeta} =\sum_{i=2}^5\frac{a_{i,10-i}}{(1-\zeta^i)(1-\zeta^{10-i})}+\alpha\frac{1+\zeta}{(1-\zeta)^2}.
\]
This gives the first three equalities. The last equality follows from the topological Lefschetz formula 
\[
\chi(\Fix(\sigma))=\displaystyle\sum_{i=0}^4(-1)^{i}\tr\left(\sigma^*\Big|_{H^{i}(X,\R)}\right)
\]
since  $\chi(\Fix(\sigma))=n_\sigma+2\alpha$ and 
since 
\[
\tr(\sigma^*|_{H^2(X,\R)})=r+l(\zeta^3+\zeta^6)+s(\zeta+\zeta^2+\zeta^4+\zeta^5+\zeta^7+\zeta^8)=r-l,
\]
 and   $\tr(\sigma^*|_{H^0(X,\R)})=\tr(\sigma^*|_{H^4(X,\R)})=1$.
\end{proof}

Finally we observe that, by Lemma \ref{tree}, the action of $\sigma$ 
on a tree of smooth rational curves is as in Figure \ref{arbol},
where double curves are in $\Fix(\sigma)$.
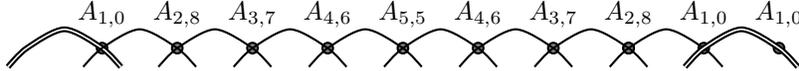
\begin{figure}[h]
\centering
\begin{tikzpicture}[xscale=.5,yscale=.5, thick] 
  \node [circle, draw, fill=black!50, inner sep=0pt, minimum width=4pt]() at (2,0) {};
  \node [circle, draw, fill=black!50, inner sep=0pt, minimum width=4pt]() at (4,0) {};
  \node [circle, draw, fill=black!50, inner sep=0pt, minimum width=4pt]() at (6,0) {};
  \node [circle, draw, fill=black!50, inner sep=0pt, minimum width=4pt]() at (8,0) {};
  \node [circle, draw, fill=black!50, inner sep=0pt, minimum width=4pt]() at (10,0) {};
  \node [circle, draw, fill=black!50, inner sep=0pt, minimum width=4pt]() at (12,0) {};
  \node [circle, draw, fill=black!50, inner sep=0pt, minimum width=4pt]() at (14,0) {};
  \node [circle, draw, fill=black!50, inner sep=0pt, minimum width=4pt]() at (16,0) {};
  \node [circle, draw, fill=black!50, inner sep=0pt, minimum width=4pt]() at (16,0) {};
  \node [circle, draw, fill=black!50, inner sep=0pt, minimum width=4pt]() at (18,0) {};
  \node [circle, draw, fill=black!50, inner sep=0pt, minimum width=4pt]() at (20,0) {};

\node[above] at (2,0.3) {$A_{1,0}$};
\node[above] at (4,0.3) {$A_{2,8}$};
\node[above] at (6,0.3) {$A_{3,7}$};
\node[above] at (8,0.3) {$A_{4,6}$};
\node[above] at (10,0.3) {$A_{5,5}$};
\node[above] at (12,0.3) {$A_{4,6}$};
\node[above] at (14,0.3) {$A_{3,7}$};
\node[above] at (16,0.3) {$A_{2,8}$};
\node[above] at (18,0.3) {$A_{1,0}$};
\node[above] at (20,0.3) {$A_{1,0}$};
    
    \draw [double] plot [smooth] coordinates {(-0.5,-0.5) (0,0) (1,0.5) (2,0) (2.5,-0.5)};
    \draw plot [smooth] coordinates {(1.5,-0.5) (2,0) (3,0.5) (4,0) (4.5,-0.5)};
    \draw plot [smooth] coordinates {(3.5,-0.5) (4,0) (5,0.5) (6,0) (6.5,-0.5)};
    \draw plot [smooth] coordinates {(5.5,-0.5) (6,0) (7,0.5) (8,0) (8.5,-0.5)};
    \draw plot [smooth] coordinates {(7.5,-0.5) (8,0) (9,0.5) (10,0) (10.5,-0.5)};
    \draw plot [smooth] coordinates {(9.5,-0.5) (10,0) (11,0.5) (12,0) (12.5,-0.5)};
    \draw plot [smooth] coordinates {(11.5,-0.5) (12,0) (13,0.5) (14,0) (14.5,-0.5)};
    \draw plot [smooth] coordinates {(13.5,-0.5) (14,0) (15,0.5) (16,0) (16.5,-0.5)};
    \draw plot [smooth] coordinates {(15.5,-0.5) (16,0) (17,0.5) (18,0) (18.5,-0.5)};
    \draw [double]plot [smooth] coordinates {(17.5,-0.5) (18,0) (19,0.5) (20,0) (20.5,-0.5)};

\end{tikzpicture}
\caption{The action of $\sigma$ on a tree of rational curves}\label{arbol}
\end{figure}

\section{Proof of the main Theorem}

We start proving some preliminary results.

\begin{proposition}\label{lemaB}
The genus of a smooth curve in the fixed locus of an order 
nine non-symplectic automorphism $\sigma$ of a K3 surface is either $0$ or $1$.
Moreover, if  $\tau=\sigma^3$ has invariants $(n,k,g)=(4,4,3)$ then the fixed locus contains a curve of genus $0$. 
\end{proposition}

\begin{proof}
Let $\sigma$ be an automorphism as in the statement.  
By Table \ref{tau} the genus of a curve fixed by $\tau=\sigma^3$ is at most $4$.
Moreover, by Lemma \ref{lemaA} we have that $\alpha=(1-g_{\sigma})+k_{\sigma}-1\geq 0$,     
thus when the invariants of $\tau$ are either $(n,k,g)=(1,1,3)$ or $(1,2,4)$, 
the automorphism $\sigma$ can not fix the curve of positive genus. 

Assume that $\sigma$ fixes a smooth curve $C$ of genus two, i.e. $\tau$ 
has invariants $(n,k,g)=(4,3,2)$. 
The linear system $|C|$ is base point free and defines a degree two 
morphism $\pi:X\to \P^2$ which can be factorized as the composition of a birational 
morphism $\theta:X\to X'$ contracting all smooth rational curves orthogonal to $C$ 
and a double cover $u:X'\to \P^2$ branched along a reduced plane sextic $S$ \cite{SD}.
Since $X'$ has rational double points, then $S$ has simple singularities, 
in particular either double or triple points \cite[III, \S 7]{BP}.
Since $|C|$ is invariant for $\sigma$, then $\sigma$ induces an order nine automorphism 
$\bar\sigma$ of $\P^2$ which preserves the sextic $S$. 
Moreover, since $\sigma$ fixes $C$ pointwise, $\bar\sigma$ fixes pointwise the line $\pi(C)$ in $\P^2$.
Thus, up to a coordinate change we can assume that $\bar\sigma(x_0,x_1,x_2)=(\zeta x_0,x_1,x_2)$, 
where $\zeta$ is a primitive $9$-th root of unity. 
Since there are no reduced plane sextics with simple singularities which are invariant 
for $\bar\sigma$, this case does not occur.

We now consider the case when $\tau$ has invariants $(n,k,g)=(4,4,3)$. 
By \cite[Theorem 3.3]{ArtebaniSarti} the fixed lattice $L^\tau$ 
of $\tau=\sigma^3$ is isomorphic to $U\oplus E_8$.  
Since $L^\tau$ is unimodular we have that $\Pic(X)=L^\tau\oplus M$,
where $M$ is a negative definite lattice.
This implies that $X$ has a $\tau$-invariant elliptic fibration 
$\pi:X\to \P^1$ with a section $E$ having a  reducible fiber $F_0$
of type $II^*$ (see the proof of \cite[Proposition 3]{ArtebaniSarti4}) 
and \cite[Lemma 3.1]{Kondo}). 
The genus three curve $C$  fixed by $\tau$ clearly intersects all fibers of $\pi$, thus 
$\tau$ preserves each fiber. This implies that $E$ is fixed by $\tau$ and 
$C$ intersects the general fiber of $\pi$ in two points 
by the Riemann-Hurwitz formula. 
By \cite[Lemma 5]{ArtebaniSarti4}, since $\pi$ has a section $E$, 
$C$ is invariant for $\sigma$, $C^2=4$ and $C\cdot E=0$,
we obtain that $f\cdot \sigma^*(f)\leq 1$, where $f$ denotes the 
class of a fiber of $\pi$. 
This implies that $\sigma^*(f)=f$, i.e. $\pi$ is invariant for $\sigma$.
In particular $\sigma$ preserves the section $E$.
Moreover $\sigma$ acts with order $3$ on the basis of $\pi$,
since a smooth genus one curve does not have automorphisms 
of order $9$ fixing points. In particular $\sigma$ does not fix $C$.
The fiber $F_0$ is also invariant for $\sigma$ (since otherwise $\pi$ would have 
three fibers of type $II^*$, which is impossible for a K3 surface),
in particular $\sigma$ must fix the component of multiplicity $6$ of $F_0$,
since it has genus zero and contains at least three fixed points. 
\end{proof}

We now study the case when $\tau=\sigma^3$ fixes a smooth curve of genus one.

\begin{proposition}\label{elliptic}
Let $\sigma$ be an automorphism of order $9$ of a K3 surface $X$ 
such that $\tau=\sigma^3$ fixes pointwise an elliptic curve $E$.
Then $|E|$ defines a $\sigma$-invariant elliptic fibration $\pi:X\to \P^1$ such that 
$\sigma$ induces an order nine automorphism of $\P^1$ with two fixed points corresponding 
to the fiber $E$ and to a singular fiber $E'$.
Moreover one of the following holds: 

\begin{enumerate}
\item $E'$ is of type $I_0^*$ and $\Fix(\tau)=E\cup F\cup\{p_1,\dots,p_4\}$,
where $F$ is the central component of $E'$ and the $p_i$'s belong each to one of the four branches of $E'$.
Moreover there are the following possibilities for $\Fix(\sigma)$:

\begin{enumerate}[\rm D1.] 
\item  $\Fix(\sigma)=F\cup\{p_1,\dots,p_4,q_1,q_2,q_3\}$ where $q_1,q_2,q_3\in E$,
\item $\Fix(\sigma)=E\cup \{q_4,q_5,p_4\}$, where $q_4,q_5\in F$, 
\item  $\Fix(\sigma)=\{q_1,\dots,q_5,p_4\}$, where $q_1,q_2,q_3\in E$ and $q_4,q_5\in F$,
\item $\Fix(\sigma)=\{q_4,q_5,p_4\}$, where $q_4,q_5\in F$.
\end{enumerate}
\item $E'$ is of type $I_9^*$, $\Fix(\tau)=E\cup F_1\cup F_2\cup F_3\cup F_4\cup \{q_1,\dots,q_7\}$ and 
$\Fix(\sigma)= F_1\cup F_4\cup\{p_1,\dots,p_{14}\}$, where the $F_i$'s are components of the fiber  $E'$, 
$p_1,\dots,p_{11}\in E'$ and $p_{12},p_{13},p_{14}\in E$. 
\end{enumerate}
 \end{proposition}

\begin{proof}
Since $E$ is invariant for $\sigma$, the elliptic fibration $\pi$ defined by $|E|$ is $\sigma$-invariant.
The action induced by $\tau$ on the basis of the fibration is not trivial, since otherwise the action of $\tau$ at a point of $E$ 
would be the identity on the tangent space, contradicting the fact that $\tau$ is non-symplectic.
Thus $\sigma$ induces an  order nine automorphism on the basis of the fibration. 
The two fixed points of such action correspond to the fiber $E$ and to another fiber $E'$ 
which must contain all smooth rational curves and points fixed by $\tau$.
By Table \ref{tau} there are two possible cases for $\Fix(\tau)$: either $(n,k,g)=(4,2,1)$ or $(7,5,1)$.

In the first case the fiber $E'$ is of type $I_0^*$. All other possible types for the fiber can be ruled out using the fact 
that $\tau$ has only one fixed curve, four isolated points and using Lemma \ref{tree}, 
which implies that in any tree of  smooth rational curves, two fixed curves by $\tau$ have distance three (in the intersection graph).
Thus the central component of $E'$ is fixed by $\tau$ and the remaining components contain one isolated fixed point each.
The automorphism $\sigma$ can either fix $E$, act on it with three fixed points, or without fixed points, by Riemann-Hurwitz formula.
Moreover, it either fixes the central component of $E'$ and four isolated points, or it acts with order three on the central component and has 
three isolated fixed points on $E'$.
If $\sigma$ fixes the central component of $E'$, then $\alpha=1$, $a_{3,7}=3$ by Lemma \ref{lemaA} and $E$ must contain three fixed points.
Thus the only possible cases for the action of $\sigma$ are the four cases in the statement.

In the second case, using again Lemma \ref{tree},  
we find that $E'$ is of type $I_9^*$ 
with the configurations in Figure \ref{curves}, where each vertex represents a curve and double circles represent the 
curves $F_1,\dots, F_4$ pointwise fixed by $\tau$. 
Moreover, $E'$ contains $7$ isolated fixed points, where non-fixed components meet.
\begin{figure}[h]
\centering
\begin{tikzpicture}[xscale=.6,yscale=.5, thick] 
  \node[circle, draw, fill=black!50, inner sep=0pt, minimum width=4pt] (n15) at (-10,-1) {};
  \node[circle, draw, fill=black!50, inner sep=0pt, minimum width=4pt] (n14) at (-10,1) {};
  \node[circle, draw, fill=black!50, inner sep=0pt, minimum width=4pt,double] (n13) at (-9,0) {};
  \node[circle, draw, fill=black!50, inner sep=0pt, minimum width=4pt] (n12) at (-8,0) {};
  \node[circle, draw, fill=black!50, inner sep=0pt, minimum width=4pt] (n11) at (-7,0) {};
  \node [double,circle, draw, fill=black!50, inner sep=0pt, minimum width=4pt](n10) at (-6,0) {};
  \node [circle, draw, fill=black!50, inner sep=0pt, minimum width=4pt](n9) at (-5,0) {};
  \node[circle, draw, fill=black!50, inner sep=0pt, minimum width=4pt] (n8) at (-4,0) {};
  \node[circle, draw, fill=black!50, inner sep=0pt, minimum width=4pt] [double](n1) at (-3,0) {};
  \node[circle, draw, fill=black!50, inner sep=0pt, minimum width=4pt] (n2) at (-2,0) {};
  \node[circle, draw, fill=black!50, inner sep=0pt, minimum width=4pt] (n3) at (-1,0) {};
  \node [circle, draw, fill=black!50, inner sep=0pt, minimum width=4pt, double](n4) at (0,0) {};
  \node [circle, draw, fill=black!50, inner sep=0pt, minimum width=4pt](n6) at (1,1) {};
  \node [circle, draw, fill=black!50, inner sep=0pt, minimum width=4pt](n7) at (1,-1) {};

  \foreach \from/\to in {n14/n13, n15/n13,n13/n12, n12/n11, n11/n10, n10/n9, n9/n8, n8/n1, n1/n2,n2/n3,n3/n4,n4/n6,n4/n7}
    \draw (\from) -- (\to);
    
\node[above] at (-9,0.2) {$F_1$};
\node[above] at (-6,0.2) {$F_2$};
\node[above] at (-3,0.2) {$F_3$};
\node[above] at (0,0.2) {$F_4$};
\end{tikzpicture}
\caption{Configuration $I_9^*$} \label{curves}
\end{figure}
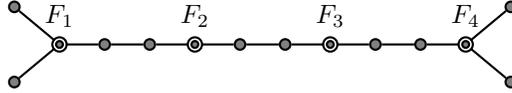

Since the intersection graph of $E'$ has no order $3$ symmetry, 
the automorphism $\sigma$ preserves each component of $E'$. 
Since the curves $F_1, F_4$ contain three fixed points each, then $\sigma$ must fix them pointwise. 
Moreover, it fixes $11$ isolated points in $E'$ (where non-fixed components meet), $4$ of them   
on $F_2$ and $F_3$. By Lemma \ref{lemaA}, since $a_{3,7}=2\alpha+1=5$, $\sigma$ also has fixed points in $E$.
By Riemann-Hurwitz formula, $\sigma$ has $3$ fixed points in $E$.
 \end{proof}

We now prove a remark about order three automorphisms of hyperelliptic curves.

\begin{lemma}\label{hyp}
Let $C$ be a hyperelliptic curve of genus $g\geq 2$ and let $f$ 
be an automorphism of $C$ or order 3, then $f$ has $2,3$ or $4$ fixed points.
\end{lemma}

\begin{proof}
Let $\pi:C\to\P^1$ be the quotient by the hyperelliptic involution. 
Since the hyperelliptic involution commutes with $f$, 
then $f$ induces an automorphism $\bar{f}$ of order $3$ on $\P^1$. 
By Riemann-Hurwitz formula $\bar{f}$ fixes $2$ points $q_1,q_2\in \P^1$.
 Thus, since $\deg(\pi)$ and $\deg(f)$ are coprime, 
 the fixed locus of $f$ is equal to $\pi^{-1}(\{q_1,q_2\})$.  
 \end{proof}

\begin{proof}[Proof of the main Theorem]

Part $(i)$ follows from Lemma \ref{tau}.
We now analyse the cases for $\Fix(\tau)$ as they appear in Table \ref{tau}
and deduce the possibilities for $\Fix(\sigma)$ according to Lemma \ref{lemaA}.

\begin{enumerate}[A.]
\item In this case $\Fix(\tau)=C\sqcup\{p\}$ with $g(C)=3$.
By Proposition \ref{lemaB}, $C$ is not contained in $\Fix(\sigma)$, thus $\alpha=0$
and we obtain two possible cases:\\ 
 $(a_{2,8}, a_{3,7}, a_{4,6}, a_{5,5})=(0,1,4,1)$ (case A1) and
 $(a_{2,8}, a_{3,7}, a_{4,6}, a_{5,5})=(1,1,1,0)$ (case A2).

\item In this case $\Fix(\tau)=C\sqcup E\sqcup \{p\}$, with $g(C)=4$ and $g(E)=0$. 
By \cite[Corollary 4.3]{ArtebaniSarti} the curve $C$ is hyperelliptic.
Since  $a_{2,8}+a_{5,5}\leq n=1$, then $\alpha=0$, i.e.
$\sigma$ has no fixed curves.
Moreover $a_{4,6}+3a_{2,8}=4$.
The case $a_{4,6}=1$ does not occur since $\Fix(\sigma)$ 
would contain only three isolated points, 
two of them on $E$ by Riemann-Hurwitz formula, 
and no one on $C$, contradicting Lemma \ref{hyp}.
Thus the only possible case is  $(a_{2,8}, a_{3,7}, a_{4,6}, a_{5,5})=(0,1,4,1)$.

\item In this case $\Fix(\tau)= E\sqcup \{p_1,\ldots,p_4\}$, with $g(E)=0$.  
Thus $\alpha$ is either $0$ or $1$.
The latter case can not occur since it would give $a_{3,7}=3$,
while $E$ contains exactly two fixed points.
 Thus $\alpha=0$ and the only solution of \eqref{star} is 
$(a_{2,8}, a_{3,7}, a_{4,6}, a_{5,5})=(1,1,1,0)$.

\item In this case $\Fix(\tau)=C\sqcup E\sqcup \{p_1,\ldots,p_4\}$ with $g(C)=1$ and $g(E)=0$.
By Proposition \ref{elliptic} there are four possible cases D1, D2, D3, D4 for the fixed locus of $\sigma$. 
The number of isolated fixed points of each type  can be computed as before 
by means of Lemma \ref{lemaA}.

\item In this case $\Fix(\tau)=C\sqcup E_1\sqcup E_2\sqcup \{p_1,\ldots,p_4\}$ 
with $g(C)=2$ and $g(E_1)=g(E_2)=0$.
By Proposition \ref{lemaB} the curve $C$ is not fixed by $\sigma$.
Moreover, by Riemann-Hurwitz formula and Lemma \ref{hyp}, 
$C$ contains exactly $4$ fixed points for $\sigma$.
Since $a_{2,8}+a_{5,5}\leq n=4$, then $\alpha$ is either $0$ or $1$.
 If $\alpha=0$, then $E_1,E_2$ should contain two isolated fixed points each.
 By Riemann-Hurwitz formula for $C$ and Lemma \ref{hyp}, 
 it should be $a_{3,7}+a_{4,6}=8$, and there is no solution of    
 \eqref{star} satisfying these conditions.
Thus we can assume $\alpha=1$ and $E_1\subseteq \Fix(\sigma)$.
Since $a_{3,7}+a_{4,6}=6$ the only solution of   \eqref{star} is
$(a_{2,8}, a_{3,7}, a_{4,6}, a_{5,5})=(3,3,3,1)$.

\item In this case $\Fix(\tau)=C \sqcup E_1 \sqcup E_2 \sqcup E_3 \sqcup \{p_1,\ldots,p_4\}$, 
where $g(C)=3$ and $g(E_i)=0$ for $i=1,2,3$.
By Proposition \ref{lemaB}, $C$ is not fixed by $\sigma$ and $\alpha\geq 1$.
Moreover $C$ is hyperelliptic by \cite[Corollary 4.3]{ArtebaniSarti}, thus by 
Riemann-Hurwitz formula and Lemma \ref{hyp} it contains exactly $2$ fixed points.
The cases $\alpha=2$ or $3$ are not possible since in both cases $a_{3,7}$ would be bigger 
than the number of fixed points on the curves $C,E_i$.
If $\alpha=1$, then  $a_{3,7}+a_{4,6}=6$ and 
the only possible solution of \eqref{star}  
is $(a_{2,8}, a_{3,7}, a_{4,6}, a_{5,5})=(3,3,3,1)$.

\item In this case $\Fix(\tau)= E_1\sqcup E_2\sqcup E_3\sqcup E_4\sqcup \{p_1,\ldots,p_7\}$, 
where $g(E_i)=0$ for $i=1,2,3,4$. 
Observe that $\alpha\in\{0,1,2,3,4\}$ 
counts how many curves among the $E_i$'s are fixed by $\sigma$.
We first assume that all the $E_i$'s  are preserved by $\sigma$,
thus $a_{3,7}+a_{4,6}=2(4-\alpha)$.
Under this condition, the system \eqref{star}  
has no solution whenever $\alpha=0,2,3,4$. 

If $\alpha=1$ the only solution of \eqref{star} is
$(a_{2,8}, a_{3,7}, a_{4,6}, a_{5,5})=(3,3,3,1)$ (case G1).  
Now assume that three of the rational curves $E_i$ are permuted by $\sigma$.
The case $\alpha=1$ is incompatible with \eqref{star}, while if $\alpha=0$  
the only solution is  
$(a_{2,8}, a_{3,7}, a_{4,6}, a_{5,5})=(1,1,1,0)$ (case G2).

\item In this case $\Fix(\tau)= C\sqcup E_1\sqcup \dots \sqcup E_4\sqcup \{p_1,\ldots,p_7\}$, 
where $g(C)=1$ and $g(E_i)=0$ for $i=1,2,3,4$. 
By Proposition \ref{elliptic}  $\sigma$ fixes two curves of genus zero and has $14$ isolated points, $3$ of them on $C$ and the others equally distributed on the two rational curves fixed by $\tau$.
Thus $(a_{2,8},a_{3,7}, a_{4,6}, a_{5,5})=(6,5,2,1)$.
\end{enumerate}
In each case the values of $r,l$ can be computed by means of Lemma \ref{lemaA} and using the fact that $r+2l=22-2m$.
The existence part of the statement will be proved in the following section.
\end{proof}

\section{Examples}

We now provide examples for all cases in Table \ref{tabla1}.
As before, $\zeta$ denotes a primitive $9$-th root of unity.

\begin{example}[Case A1 and its descendents]\label{A1}

Let $F_1,F_4\in \C[x_0,x_1]$ be general homogeneous polynomials of degree $1$ and $4$ respectively. 
The following is a smooth $K3$ surface:
 \[
 S=\{F_4(x_0,x_1)+F_1(x_0,x_1)x_2^3+x_2x_3^3=0\}\subset  \P^3
 \]
with the automorphism $\sigma(x_0,x_1,x_2,x_3)=(x_0,x_1,\zeta^6 x_2,\zeta^4 x_3)$.
The fixed locus of $\tau=\sigma^3$ is the union of the curve $C=S\cap\{x_3=0\}$ of genus $3$ and the point $p_1=(0,0,0,1)$. 
Thus $(n,k,g)=(1,1,3)$.
The fixed locus of $\sigma$ contains exactly $6$ points, five of them on the curve $C$ 
(the four roots of $F_4(x_0,x_1)$ and the point $q_1=(0,0,1,0)$), thus we are in {\bf case A1}.
\end{example}

\begin{example}[Case A2 and its descendents]\label{A2}
For a general choice of  the coefficients the following is a $K3$ surface:
\[
S=\{ax_0^2x_1x_2+bx_1^2x_2^2+cx_2^3x_0+dx_1^3x_0+fx_0^4+x_2x_3^3=0\}\subseteq \P^3
\]
which carries the automorphism $\sigma(x_0,x_1,x_2,x_3)=(x_0,\zeta^3 x_1,\zeta^6x_2,\zeta x_3)$. 
The fixed locus of $\tau$ is the union of the curve $C=S\cap\{x_3=0\}$ of genus $3$ and the point $p_1=(0,0,0,1)$,
thus $(n,k,g)=(1,1,3)$.  The fixed locus of $\sigma$ contains $p_1$ and the points $p_2=(0,1,0,0), p_3=(0,0,1,0)$.
Thus we are in {\bf case A2}.

For special values of the coefficients we obtain other examples.
In each case we assume that the coefficients are general with the given assumption.
\begin{enumerate}

\item if $b=0$, the curve $C$ is the union of a smooth cubic $E$ and a line $L$ 
and the surface $S$ has three $A_2$ singularities in $L\cap E$. 
The minimal resolution $\pi:\tilde S\to S$ is a K3 surface and $\sigma$ lifts to 
an  automorphism $\tilde \sigma$ on $\tilde S$.
Each exceptional divisor over the three singular points  contains a fixed point
 for $\tilde\tau$ and $\tilde p_1=\pi^{-1}(p_1)\in\Fix(\tilde\tau)$.
Thus the invariants for $\tilde \tau$ are $(n,k,g)=(4,2,1)$. 
Since the three exceptional divisors of type $A_2$ are permuted by 
$\tilde\sigma$, $\Fix(\tilde\sigma)$ only contains the preimages $\tilde p_i$ 
of the points $p_i$, $i=1,2,3$, thus 
this is an example of {\bf case D4}.

\item if $a=-d-2f, b=f-d, c=d$ 
the curve $C$ acquires three nodes and the surface $S$ has three $A_2$ singularities.
Thus $\tilde\tau$ fixes three points, one in each of the exceptional divisors, 
the point $\tilde p_1$ and the proper transform of $C$, which has genus $0$.
Thus the invariants for $\tau$ are $(n,k,g)=(4,1,0)$ and this is an example of {\bf case C}. 

\item if $c=0$ the curve $C$ acquires a tacnode at $p_3=(0,0,1,0)$, 
which gives a singularity of type $E_6$ of the surface. 
In this case  $\Fix(\tilde\tau)$ contains the proper transform of $C$, which has genus $1$,
a rational curve and three points in the exceptional divisor of type $E_6$, and $\tilde p_1$. Thus $(n,k,g)=(4,2,1)$.
The automorphism $\tilde \sigma$ preserves the exceptional divisor of type 
$E_6$ and thus fixes its central component.  
By Theorem \ref{main}  this is an example of {\bf case D1}.
\end{enumerate}
\end{example}

\begin{example}[Case B and its descendents]\label{B}
Consider the elliptic surface with Weierstrass equation $$y^2=x^3+t(t^3-a)(t^3-b)(t^3-c), \mbox{ for } a,b,c\in\C,$$ 
which carries the automorphism 
$\sigma(x, y, t)=(\zeta^4x,\zeta^6 y,\zeta^3 t)$.
For general $a,b,c\in \C$ this defines a K3 surface. More precisely, if $a,b,c$ are distinct and non-zero, then 
by \cite[Table IV.3.1]{Miranda}  the elliptic fibration has a singular fiber of type $II$ 
for $t=0$, of type $IV$ for $t=\infty$ and nine fibers of type $II$ over the zeroes of $P(t)=(t^3-a)(t^3-b)(t^3-c)$.
The automorphism $\tau$ preserves each fiber of the elliptic fibration. 
The fixed locus of $\tau$ clearly contains the curve $C=\{x=y^2-tP(t)=0\}$, which is hyperelliptic of genus $4$ and is a $2$-section of the fibration, 
and the section at infinity $S_{\infty}$. Moreover, since neither $C$ or $S_{\infty}$ passes through the center $p$ of the fiber of type $IV$,
then $p$ is an isolated fixed point of $\tau$. Thus this gives a family of examples of {\bf case B}.
Observe that $\sigma$ only preserves the fibers over $t=0$ and $t=\infty$. 
This implies that its fixed locus consists of $4$ points in the fiber  over $t=\infty$ (the center $p$ and the points where $S_{\infty}$ and $C$ intersect) and two points on the fiber over $t=0$ (where $S_{\infty}$ and $C$ intersect).

For special values of $a,b,c\in \C$ we obtain other examples:
\begin{enumerate}
\item $a=0$: the fibration has one fiber of type $IV^*$ over $t=0$, $IV$ over $t=\infty$ 
and six fibers of type $II$ over the zeros of $(t^3-b)(t^3-c)$. The fixed locus of $\tau$ in this case contains the curve $C$, which has genus two, the section at infinity $S_{\infty}$, the central component and three isolated points of the fiber of type $IV^*$ and the center of the fiber of type $IV$. Thus this gives a family of examples of {\bf case E}.
\item $b=c$: the fibration has four fibers of type $II$ over $t=0$ and the zeroes of $t^3-a$, and 
four fibers of type $IV$ over $t=\infty$ and the zeroes of $t^3-b$. 
The fixed locus of $\tau$ contains the curve $C$, which has genus one, the section at infinity $S_{\infty}$ and the four centers of the fibers of type $IV$. Thus we are in case $D$. 
The fixed locus of $\sigma$ contains no curves and six isolated points, four on the fiber over $t=\infty$ (the center and the intersection points with $C$ and $S_{\infty}$) and two points on the fiber of type $II$ (where $C$ and $S_{\infty}$ intersect). This gives a family of examples of {\bf case D3}.

\item $c=\infty$: in this case the equation of the fibration is $y^2=x^3+t(t^3-a)(t^3-b)$, so that there are $7$ fibers of type $II$, over $t=0$ and over the zeroes of $(t^3-a)(t^3-b)$, and one fiber of type $II^*$ over $t=\infty$. The fixed locus of $\tau$ contains the curve $C$, which has genus $3$, the section at infinity and two rational curves in the fiber of type $II^*$. Moreover it fixes $4$ points in the fiber of type $II^*$.
Thus we obtain a family of examples of {\bf case F}. 

\item $a=0, b=c$: the fibration has a fiber of type $IV^*$ over $t=0$, of type $IV$ over $t=\infty$ and three fibers of type $IV$ over the zeroes of 
$t^3-b$. Observe that in this case $C$ splits in the union of two sections of the fibration: 
\[
C_1=\{x=y-t^2(t^3-b)=0\},\quad C_2=\{x=y+t^2(t^3-b)=0\}.
\]
The fixed locus of $\tau$ contains the curves $C_1, C_2$, the section at infinity $S_{\infty}$ and the central component of the fiber of type $IV^*$.
Moreover, $\tau$ fixes the centers of the four fibers of type $IV$. 
Observe that $\sigma$ must preserve each component of the fiber over $t=0$, so that it fixes its central component.
Thus we obtain a family of examples of {\bf case G1}.  

\item $a=0, c=\infty$: in this case the equation of the fibration is $y^2=x^3+t^4(t^3-b)$, so that there is a fiber of type $IV^*$ over $t=0$,
of type $II^*$ over $t=\infty$ and three fibers of type $II$ over the zeroes of $t^3-b$.
The fixed locus of $\tau$ contains the curve $C$, which has genus $1$, the section at infinity, two rational curves in the fiber of type $II^*$ and one rational curve in the fiber of type $IV^*$. Moreover it has $4$ fixed points in the fiber of type $II^*$ and $3$ fixed points in the fiber of type $IV^*$.
 Thus we obtain a family of examples of {\bf case H}. 
\end{enumerate}

\end{example}

\begin{example}[Case G2]
Consider the elliptic K3 surface $X$ with Weierstrass equation $y^2=t^4(t^3-b)^2$ (the case (iii) in Example \ref{B}).
An explicit computation shows that, fixed $S_{\infty}$ as zero section, $C_1$ is a $3$-torsion section and $C_2=C_1\oplus C_1$, 
where $\oplus$ denotes the sum in the group law of the  elliptic curve over $\C(t)$ associated to the fibration. 
The translation by the section $C_1$ defines a symplectic automorphism $\varphi$ of order three on $X$. A computation with Magma \cite{Magma} available at \url{https://bit.ly/2FwQ0ZU} gives that
\[
\sigma'(x,y,t)=(\alpha(x,y,t), \beta(x,y,t),t),
\]
where
\[
\alpha(x,y,t)=-2\frac{(t^5+t^2)}{x^2}(y-(t^5+t^2))
\]
\[
\beta(x,y,t)=-4\frac{(t^5+t^2)^2}{x^3}(y-(t^5+t^2))+(t^5+t^2).
\]
Since $\sigma'$  commutes with $\sigma$, as can be checked directly,
then  $\sigma':=\varphi\circ \sigma$ is a non-symplectic automorphism of order $9$ on $X$ such that $(\sigma')^3=\sigma^3=\tau$.
The sections $S_{\infty},C_1,C_2$ are permuted by $\sigma'$. This implies that the three branches of the fiber of type $IV^*$ are also permuted by $\sigma$ while the central component, which is $\sigma'$-invariant, contains two isolated fixed points of $\sigma'$. 
Finally, $\sigma'$ acts on the fiber of type $IV$ permuting the three components and fixing the center.
Thus $\sigma'$ fixes exactly three isolated points, giving an example of {\bf case G2}.
\end{example}

\begin{example}[Case D2]
Consider the elliptic fibration
\[
y^2=x^3+x+t^9+c,\quad c\in\C
\]
with the automorphism $\sigma(x,y,t)=(x,y,\zeta t)$.
The fibration has a singular fiber of type $I_0^*$ over $t=\infty$
and $9$ fibers of type $I_1$ over the zeroes of $t^9+c=0$.
Observe that $\tau$ preserves the fibers over $t=0$ and $t=\infty$, 
thus $\Fix(\tau)$ contains a curve of genus $1$ and the central curve of the fiber $I_0^*$. 
Moreover, the fiber of type $I_0^*$ contains $4$ isolated fixed points for $\tau$, so that the invariants for $\tau$ are $(n,k,g)=(4,2,1)$.
The automorphism $\sigma$ on the fiber over $t=0$ acts as the identity, 
thus this corresponds to {\bf case D2}.  
\end{example}

\begin{remark}
Examples for members of the families B, E, F and for H have been given also by Taki in \cite{Taki}.
The case A1 is missing in \cite[Theorem 1.3]{Taki}, as was also observed in \cite[Remark 6.7]{Brandhorst}.
\end{remark}

\section{Moduli}

In this section we will consider the families of K3 surfaces with a non-symplectic automorphism of order $9$ 
of maximal dimension, i.e. those of type $A_1, A_2$ and $B$.
We show that their moduli space is irreducible and give birational 
maps to moduli spaces of curves with automorphisms.

\begin{proposition}\label{irr}
Let $X$ be a K3 surface with a non-symplectic automorphism $\sigma$ of order $9$ 
such that $\Pic(X)=L^{\tau}$, where $\tau=\sigma^3$. If $\Fix(\sigma)$ is of
\begin{enumerate}
\item  type A1, then up to isomorphism $(X,\langle \sigma\rangle)$ is as in Example \ref{A1},
\item   type A2, then up to isomorphism $(X,\langle\sigma\rangle)$ is as in Example \ref{A2},
\item    type B, then up to isomorphism $(X,\langle\sigma\rangle)$ is as in Example \ref{B}.
\end{enumerate}
\end{proposition}

\begin{proof}
We first assume that $\tau=\sigma^3$ has invariants $(n,k,g)=(1,1,3)$.
Let $C$ be the smooth genus three curve in the fixed locus of $\tau$.
By \cite[Proposition 4.9]{ArtebaniSarti} the linear system $|C|$ 
defines an embedding $\varphi_C:X\to \P^3$. 
After identifying $X$ with its image in $\P^3$ we can assume that
\[
\tau(x_0,x_1,x_2,x_3)=(x_0,x_1,x_2,\zeta^3 x_3).
\]
Since $C$ is invariant for $\sigma$, then  
$\sigma$ is given by a projectivity of $\P^3$ which 
 leaves invariant the hyperplane $x_3=0$  and the point $p=(0,0,0,1)$.
An order three automorphism of $H\cong\P^2$ is exactly one of the following up to a coordinate change 
 and up to taking powers:
 \[
 f_1(x_0,x_1,x_2)=(x_0,x_1,\zeta^3 x_2),\quad f_2(x_0,x_1,x_2)=(x_0,\zeta^3 x_1,\zeta^{6} x_2).
 \]
 This implies that, up to coordinate changes and up to taking powers, $\sigma$ is one of the following:
 \[
 \sigma_{1,k}(x_0,x_1,x_2,x_3)=(x_0,x_1,\zeta^3 x_2,\zeta^k x_3),\  \sigma_{2}(x_0,x_1,x_2,x_3)=(x_0,\zeta^3 x_1,\zeta^{6}x_2,\zeta x_3),
 \]
 where $k=1,2$.
 Analyzing each of these automorphisms we obtain that the only ones which allow a smooth homogeneous 
 invariant polynomial of degree $4$ are  $\sigma_{1,2}$ and $\sigma_{2}$.
 In case the automorphism is $\sigma_{1,2}$ we find that the family of smooth invariant polynomials 
 is the one in Example \ref{A1};  in case it is  $\sigma_{2}$ the family of smooth invariant polynomials 
 is the one in Example \ref{A2}.
 
 We now assume that $\sigma$ is of type $B$. 
 By \cite[Proposition 4.2]{ArtebaniSarti} and its proof, $X$ has an elliptic fibration $\pi:X\to \P^1$ with Weierstrass form
 \[
 y^2=x^3+p(t),
 \]
where $\deg(p)=10$ and in this model $\tau(x,y,t)=(\epsilon x,y,t)$. 
The fibration has a reducible fiber of type $IV$ over $t=\infty$ 
and $10$ singular fibers of type $II$ over the zeroes of $p$.
An argument similar to the one in the proof of Lemma \ref{B} shows
that $\sigma$ preserves this fibration, using \cite[Lemma 5]{ArtebaniSarti4}.  
Observe that $\sigma$ induces an order three automorphism on the basis of the fibration 
with two fixed points, one of them in $t=\infty$.
Up to a coordinate change in $t$ we can assume that $\sigma$ acts as $t\mapsto \epsilon t$ on the basis of $\pi$.
This implies that up to a constant $p(t)=t(t^3-a)(t^3-b)(t^3-c)$ for distinct $a,b,c\in \C$, which gives the family in Example \ref{B}.
 \end{proof}
 
Let $(X,\sigma)$ be a K3 surface with a non-symplectic automorphism of order $9$ 
such that $\sigma^*(\omega_X)=\zeta\omega_X$,
fix an isometry  $\varphi: H^2(X,\Z)\to L_{K3}$,  let $\rho=\varphi^{-1}\sigma^*\varphi$, 
 $\iota=\rho^3$, $M\subseteq L_{K3}$ be the fixed lattice of $\iota$ and $N$ be its orthogonal complement.
By \cite[\S 11]{DK} the moduli space of $(M,\rho)$-polarized K3 surfaces is isomorphic 
to a complex ball quotient $\mathcal D/\Gamma$ where
\[
\mathcal D=\{z\in \P(V): \langle z,\bar z\rangle >0\}\cong \mathbb B_{d},\quad 
\Gamma=\{\gamma\in O(N): \gamma \circ \rho=\rho \circ \gamma\}
\]
with $V = \{z\in N\otimes \C: \rho(z)=\zeta z\}$ and $d=\dim(V)-1=\frac{m}{3}-1$.
 The points in the moduli space which correspond to $(M,\rho)$-ample polarized K3 surfaces,
i.e. such that $\rho$ is induced by an automorphism of the surface, 
belong to an open subset defined as the quotient of the 
complement of a divisor $\Delta$ in $\mathcal D$ \cite[Lemma 11.5]{DK}.

Given a general pair $(X,\sigma)$ as in Example \ref{A1},
 we will denote  by  ${\mathcal A_1}$ the corresponding moduli space of $(M,\rho)$-polarized K3 surfaces.
Similarly we define  ${\mathcal A_2}$ and ${\mathcal B}$ as the moduli spaces corresponding 
to the general members of the families of type $A2$ and $B$. By Theorem \ref{main} and Proposition \ref{irr} 
we obtain the following.

\begin{corollary}
The moduli space of $K3$ surfaces with a non-symplectic automorphism of order $9$ 
has three irreducible components of maximal 
dimension $2$: ${\mathcal A_1},  {\mathcal A_2}$ and ${\mathcal B}$. 
\end{corollary}

\begin{remark}
The moduli space $\mathcal A_1$ is birational  to the moduli space $\mathcal M_3^3$
of genus three curves $C$ carrying a cyclic automorphism $f$ of order $3$ 
with $C/(f)\cong \P^1$. The isomorphism is given by the map
\begin{equation}\label{iso}
(X,\langle\sigma\rangle)\mapsto (C,\sigma_{|C}),
\end{equation}
where $C$ is the fixed curve of $\tau=\sigma^3$.
Conversely, given a pair $(C,f)$ as above, 
observe that $C$ is not hyperelliptic by Lemma \ref{hyp}, 
since it contains $5$ fixed points for $f$.
Thus $C$ can be embedded in $\P^2$ as a  smooth 
plane quartic and $f$ is induced by an order three automorphism of $\P^2$.
Up to a change of coordinates we can assume
\[
f(x_0,x_1,x_2)=(x_0,x_1,\epsilon x_2).
\]
Let $Y$ be the triple cover of $\P^2$ branched along $C$ and $2L$,
where $L=\{x_2=0\}$ is the line fixed by $f$.
The normalization of $Y$ is a K3 surface $X$ with an order nine automorphism $\sigma$, 
obtained  lifting $f$, which gives a pair $(X,\langle \sigma\rangle)$ in $\mathcal A_1$.

Similarly one can show that $\mathcal A_2$ is birational to the moduli space 
of non-hyperelliptic curves of genus three $C$ with a cyclic automorphism $f$ of order $3$ 
with $g(C/(f))=1$. 

Finally, the map \eqref{iso} defines an isomorphism between $\mathcal B$ and the moduli space 
of hyperelliptic curves of genus $4$ with a cyclic automorphism $f$ of order $3$.
Observe that one such curve $C$ can be defined by an equation of the form $y^2-t(t^3-as^3)(t^3-bs^3)(t^3-cs^3)=0$ 
with distinct $a,b,c\in \C$ in $\P(1,1,5)$ and in these coordinates $f(s,t,y)=(s,\epsilon t,\epsilon^2 y)$.
The triple cover of $\P(1,1,5)$ branched along $C$ is birational to a K3 surface with an order nine automorphism 
which gives a point in $\mathcal B$.
\end{remark}

\bibliographystyle{plain}
\bibliography{biblio9}

\end{document}